\RequirePackage{fix-cm}
\documentclass[smallextended]{svjour3}       
\smartqed  
\usepackage{graphicx}
\usepackage{amsmath}
\usepackage{amsfonts}       
\usepackage{url}            
\usepackage{nicefrac}       

\newcommand{\minimize}[1]{\mathop{\hbox{minimize}}_{#1}}
\newcommand{\R}{{\rm I\!R}}
\newcommand{\prox}[1]{\hbox{{\bf prox}}_{#1}}
\newcommand{\dom}[1]{\hbox{{\bf dom~}}{#1}}
\newcommand{\argmin}[1]{\mathop{\hbox{argmin}}_{#1}}

\newcommand{\interior}{\mathop{\hbox{int}}}

\newtheorem{assumption}{Assumption}

\def\norm#1{\|#1\|}
\def\half{\frac 1 2}

\usepackage{color}

\begin{document}

\title{``Active-set complexity'' of proximal gradient
\thanks{We acknowledge the support of the Natural Sciences and Engineering Research Council of Canada (NSERC) [Discovery Grant, reference numbers \texttt{\#}355571-2013, \texttt{\#}2015-06068].\\
Cette recherche a \'et\'e financ\'ee par le Conseil de recherches en sciences naturelles et en g\'enie du Canada (CRSNG) [Discovery Grant, num\'eros de r\'ef\'erence \texttt{\#}355571-2013, \texttt{\#}2015-06068].}
}
\subtitle{How long does it take to find the sparsity pattern?}


\author{Julie~Nutini \and Mark~Schmidt \and Warren~Hare}


\institute{Julie Nutini \at
             Department of Computer Science, The University of British Columbia \\
             201-2366 Main Mall,
    	     Vancouver BC, V6T 1Z4,
    	     Canada  \\
              \email{jnutini@cs.ubc.ca}           
           \and
           Mark Schmidt \at
           Department of Computer Science, The University of British Columbia \\
           \email{schmidtm@cs.ubc.ca}
           \and
           Warren Hare \at
           Department of Mathematics, The University of British Columbia Okanagan \\
           \email{warren.hare@ubc.ca}  
}

\date{Received: date / Accepted: date}

\maketitle

\begin{abstract}
Proximal gradient methods have been found to be highly effective for solving minimization problems with non-negative constraints or $\ell_1$-regularization. Under suitable nondegeneracy conditions, it is known that these algorithms identify the optimal sparsity pattern for these types of problems in a finite number of iterations. However, it is not known how many iterations this may take. We introduce the notion of the ``active-set complexity'', which in these cases is the number of iterations before an algorithm is guaranteed to have identified the final sparsity pattern. We further give a bound on the active-set complexity of proximal gradient methods in the common case of minimizing the sum of a strongly-convex smooth function and a separable convex non-smooth function.
\end{abstract}

\section{Motivation}\label{sec:motivation}
We consider the problem 
\begin{equation}\label{eq:problem}
	\minimize{x \in \R^n} \quad f(x) + g(x),
\end{equation}
where $f$ is $\mu$-strongly convex and the gradient $\nabla f$ is $L$-Lipschitz continuous.
We assume that $g$ is a separable function,
\[
g(x) = \sum_{i=1}^n g_i(x_i),
\]
and each $g_i$ only needs to be a proper convex and lower semi-continuous function (it may be non-smooth or infinite at some $x_i$). In machine learning, a common choice of $f$ is the squared error $f(x) = \half\norm{Ax-b}^2$ (or an $\ell_2$-regularized variant to guarantee strong-convexity). The squared error is often paired with a scaled absolute value function $g_i(x_i) = \lambda|x_i|$ to yield a sparsity-encouraging $\ell_1$-regularization term. This is commonly known as the LASSO problem~\cite{tibshirani1996}. The $g_i$ can alternatively enforce bound constraints (e.g., the dual problem in support vector machine optimization~\cite{cortes1995}), such as the $x_i$ must be non-negative, by defining $g_i(x_i)$ to be an indicator function that is zero if the constraints are satisfied and $\infty$ otherwise.

One of most widely-used methods for minimizing functions of this form is the proximal gradient (PG) method~\cite{levitin1966constrained,beck2009,nesterov2013,bertsekas2015convex}, which uses an iteration update given by
\[
	x^{k+1} = \prox{\frac{1}{L} g} \left(x^k - \frac{1}{L} \nabla f(x^k) \right),
\]
where the proximal operator is defined as
\[
	\prox{\frac{1}{L} g}(x) = \argmin{y} \frac{1}{2} \| y - x\|^2 + \frac{1}{L}g(y).
\]
When the proximal gradient method is applied with non-negative constraints or $\ell_1$-regularization, an interesting property of the method is that the iterations $x^k$ will match the sparsity pattern of the solution $x^*$ for all sufficiently large $k$ (under a mild technical condition). Thus, after a finite number of iterations the algorithm ``identifies'' the final set of non-zero variables. This is useful if we are only using the algorithm to find the sparsity pattern, since it means we do not need to run the algorithm to convergence. It is also useful in designing faster algorithms (for example, see~\cite{krishnan2007,curtis2015,buchheim2016} for non-negativity constrained problems and~\cite{wen2010,byrd2015,santis2015}) for $\ell_1$-regularized problems). After we have identified the set of non-zero variables we could switch to a more sophisticated solver like Newton's method applied to the non-zero variables. In any case, we should expect the algorithm to converge faster after identifying the final sparsity pattern, since it will effectively be optimizing over a lower-dimensional subspace.

The idea of finitely identifying the set of non-zero variables dates back at least 40 years to the work of Bertsekas~\cite{bertsekas1976goldstein} who showed that the projected gradient method identifies the sparsity pattern in a finite number of iterations when using non-negative constraints (and suggests we could then switch to a superlinearly convergent unconstrained optimizer). Subsequent works have shown that finite identification occurs in much more general settings including cases where $g$ is non-separable, where $f$ may not be convex, and even where the constraints may not be convex~\cite{burke1988identification,wright1993identifiable,hare2004,hare2011}. The active-set identification property has also been shown for other algorithms like certain coordinate descent and stochastic gradient methods~\cite{mifflin2002,wright2012,lee2012manifold}.

Although these prior works show that the active-set identification must happen after some finite number of iterations, they only show that this happens asymptotically. In this work, we introduce the notion of the ``active-set complexity'' of an algorithm, which we define as the number of iterations required before an algorithm is guaranteed to have reached the active-set. 
We further give bounds, under the assumptions above and the standard nondegeneracy condition, on the active-set complexity of the proximal gradient method. 
We are only aware of one previous work giving such bounds, the work of Liang et al.\ who included a bound on the active-set complexity of the proximal gradient method~\cite[Proposition~3.6]{liang2017activity}. 
Unlike this work, their result does not evoke strong-convexity.  Instead, their work applies an inclusion condition on the local subdifferential of the regularization term that ours does not require.  By focusing on the strongly-convex
 case (which is common in machine learning due to the use of regularization), we obtain a simpler analysis and a much tighter bound than in this previous work. Specifically, both rates depend on the ``distance to the subdifferential boundary'', but in our analysis this term only appears inside of a logarithm rather than outside of it.

\section{Notation and assumptions}

We assume that $f$ is $\mu$-strongly convex so that for some $\mu > 0$, we have
\[
	f(y) \ge f(x) + \langle \nabla f(x), y - x \rangle + \frac{\mu}{2} \| y - x\|^2, \quad \text{for all $x,y \in \R^n$}.
\]
Further, we assume that its gradient $\nabla f$ is $L$-Lipschitz continuous, meaning that
\begin{equation}\label{eq:lipschitz}
	\| \nabla f(y) - \nabla f(x) \| \le L \| y - x\|, \quad \text{for all $x,y \in \R^n$}.
\end{equation}
By our separability assumption on $g$, the subdifferential of $g$ is simply the concatenation of the subdifferentials of each $g_i$. Further, the subdifferential of each individual $g_i$ at any $x_i \in \R$ is defined by
\[
	\partial g_i(x_i) = \{ v \in \R : g_i(y) \ge g_i(x_i) + v \cdot (y-x_i), \text{ for all $y \in \dom{g_i}$}\},
\]
which implies that the subdifferential of each $g_i$ is just an interval on the real line. In particular, the interior of the subdifferential of each $g_i$ at a non-differentiable point $x_i$ can be written as an open interval,
\begin{equation}
\label{eq:interior}
	 \interior \partial g_i(x_i) \equiv (l_i, u_i ),
\end{equation}
where $l_i \in \R \cup \{-\infty\}$ and $u_i \in \R \cup \{\infty\}$ (the $\infty$ values occur if $x_i$ is at its lower or upper bound, respectively).

As in existing literature on active-set identification~\cite{hare2004}, we require the {\em nondegeneracy} condition that $-\nabla f(x^*)$ must be in the ``relative interior'' of the subdifferential of $g$ at the solution $x^*$. For simplicity, we present the nondegeneracy condition for the special case of \eqref{eq:problem}.
\begin{assumption}\label{assump:nondegeneracy} We assume that $x^*$ is a nondegenerate solution for problem \eqref{eq:problem}, where $x^*$ is {\em nondegenerate} if and only if 
\[
\begin{cases}
	-\nabla_i f(x^*) \!=\!  \nabla_i g(x^*_i) &\!\!\!\!\text{if $\partial g_i(x_i^*)$ \!is a singleton \!($g_i$\! smooth at $x_i^*$)} \\
	-\nabla_i f(x^*) \!\in\! \interior \partial g_i(x^*_i) &\!\!\!\!\text{if $\partial g_i(x_i^*)$ \!is not a singleton \!($g_i$\! non-smooth at $x_i^*$)}.
\end{cases}
\]
\end{assumption}
Under this assumption, we ensure that $-\nabla f(x^*)$ is in the ``relative interior" (see \cite[Section 2.1.3]{boyd2004convex}) of the subdifferential of $g$ at the solution $x^*$. In the case of non-negative constraints, this requires that $\nabla_i f(x^*) > 0$ for all variables $i$ that are zero at the solution ($x_i^* = 0$). For $\ell_1$-regularization, this requires that $|\nabla_i f(x^*)| < \lambda$ for all variables $i$ that are zero at the solution, which is again a strict complementarity condition~\cite{santis2015}.\footnote{Note that $|\nabla_i f(x^*)| \leq \lambda$ for all $i$ with $x_i^* = 0$ follows from the optimality conditions, so this assumption simply rules out the case where $|\nabla_i f(x_i^*)|=\lambda$.}

\begin{definition}
The {\em active-set} $\mathcal{Z}$ for a separable $g$ is defined as 
\[
	\mathcal{Z} = \{ i : \partial g_i(x_i^*) \text{ is not a singleton} \}.
\]
\end{definition} 
By the above definition and recalling the interior of the subdifferential of $g_i$ as defined in \eqref{eq:interior}, the set $\mathcal{Z}$ includes indices $i$ where $x_i^*$ is equal to the lower bound on $x_i$, is equal to the upper bound on $x_i$, or occurs at a non-smooth value of $g_i$. In the case of non-negative constraints and $\ell_1$-regularization under Assumption 1, $\mathcal{Z}$ is the set of non-zero variables at the solution.
Formally, the {\em active-set identification property} for this problem is that for all sufficiently large $k$ we have that $x_i^k = x_i^*$ for all $i \in \mathcal{Z}$.
An important quantity in our analysis is the minimum distance to the nearest boundary of the subdifferential \eqref{eq:interior} among indices $i \in \mathcal{Z}$. This quantity is given by
\begin{equation}
\label{eq:Delta}
	 \delta =  \min_{i\in\mathcal{Z}}\left\{ \min\{ -\nabla_i f(x^*) - l_i, u_i + \nabla_i f(x^*) \}\right\}.
\end{equation}

\section{Finite-time active-set identification}

In this section we show that the PG method identifies the active-set of \eqref{eq:problem} in a finite number of iterations. Although this result follows from the more general results in the literature, by focusing on~\eqref{eq:problem} and the case of strong-convexity we give a substantially simpler proof that will allow us to easily bound the active-set iteration complexity of the method.

Before proceeding to our main contributions, we state the linear convergence rate of the proximal gradient method to the (unique) solution $x^*$.

\begin{theorem}\cite[Prop. 3]{schmidt2011convergence}\label{thm:convergence}
Consider problem \eqref{eq:problem}, where $f$ is $\mu$-strongly convex with $L$-Lipschitz continuous gradient, and the $g_i$ are proper convex and lower semi-continuous. 
Then for every iteration $k \ge 1$ of the proximal gradient method, we have
\begin{equation}
\label{eq:linConv}
\norm{x^k - x^*} \leq \left(1-\frac{1}{\kappa}\right)^k\norm{x^0 - x^*},
\end{equation}
where $\kappa := L/\mu$ is the condition number of $f$. 
\end{theorem}

Next, we state the finite active-set identification result. Our argument essentially states that $\norm{x^k - x^*}$ is eventually always less then $\delta/2L$, where $\delta$ is defined as in \eqref{eq:Delta}, and at this point the algorithm always sets $x_i^k$ to $x_i^*$ for all $i \in \mathcal{Z}$.

\begin{lemma}\label{lem:activeset}
Consider problem \eqref{eq:problem}, where $f$ is $\mu$-strongly convex with $L$-Lipschitz continuous gradient, and the $g_i$ are proper convex and lower semi-continuous. 
Let Assumption \ref{assump:nondegeneracy} hold for the solution $x^*$. Then for any proximal gradient method with a step-size of $1/L$, there exists a $\bar{k}$ such that for all $k > \bar{k}$ we have $x_i^k = x_i^*$ for all $i \in \mathcal{Z}$.
\end{lemma}

\begin{proof} 
By the definition of the proximal gradient step and the separability of $g$, for all $i$ we have
\[
	x_i^{k+1} \in \argmin{y}\left\{ \frac{1}{2}\left| y - \left(x^k_i -  \frac{1}{L}\nabla_i f(x^k)\right) \right|^2 + \frac{1}{L}g_i(y)\right\}.
\]
This problem is strongly-convex, and its unique solution satisfies
\[
0 \in y - x_i^k + \frac{1}{L} \nabla_i f(x^k) + \frac{1}{L}\partial g_i(y),
\]
or equivalently that
\begin{equation}
\label{eq:proxSol}
L(x_i^k - y) - \nabla_i f(x^k) \in \partial g_i(y).
\end{equation}
By Theorem~\ref{thm:convergence}, there exists a minimum finite iterate $\bar{k}$ such that $\norm{x^{\bar{k}}~-~x^*} \leq \delta/2L$. Since $|x_i^k - x_i^*| \leq \norm{x^k - x^*}$, this implies that for all $k \geq \bar{k}$ we have 
\begin{equation}
\label{eq:itBound}
-\delta/2L \leq x_i^k - x_i^* \leq \delta/2L, \quad \text{for all $i$.}
\end{equation}
Further, the Lipschitz continuity of $\nabla f$ in \eqref{eq:lipschitz} implies that we also have
\begin{align*}
|\nabla_i f(x^k) - \nabla_i f(x^*)| & \leq \norm{\nabla f(x^k) - \nabla f(x^*)}\\
& \leq L\norm{x^k - x^*}\\
& \leq \delta/2,
\end{align*}
which implies that
\begin{equation}
\label{eq:gBound}
-\delta/2 - \nabla_i f(x^*) \leq - \nabla_i f(x^k) \leq \delta/2 - \nabla_i f(x^*).
\end{equation}
To complete the proof it is sufficient to show that for any $k \geq \bar{k}$ and $i \in \mathcal{Z}$ that $y = x_i^*$ satisfies \eqref{eq:proxSol}. Since the solution to~\eqref{eq:proxSol} is unique, this will imply the desired result. We first show that the left-side is less than the upper limit  $u_i$ of the interval $\partial g_i(x_i^*)$,
\begin{align*}
L(x_i^k - x_i^*) - \nabla_i f(x^k) & \leq \delta/2 - \nabla_i f(x^k) & \text{(right-side of~\eqref{eq:itBound})}\\
& \leq \delta - \nabla_i f(x^*) & \text{(right-side of~\eqref{eq:gBound})}\\
& \leq (u_i + \nabla_i f(x^*)) - \nabla_i f (x^*) & \text{(definition of $\delta$,~\eqref{eq:Delta})}\\
& \leq u_i.
\end{align*}
We can use the left-sides of~\eqref{eq:itBound} and~\eqref{eq:gBound} and an analogous sequence of inequalities to show that $L(x_i^k~-~x_i^*)~-~\nabla_i f(x^k) \geq l_i$, implying that $x_i^*$ solves~\eqref{eq:proxSol}. \qed
\end{proof}

\section{Active-set complexity}
\label{sec:maniRate}

The active-set identification property shown in the previous section could also be shown using the more sophisticated tools used in related works~\cite{burke1988identification,hare2004}. However, an appealing aspect of the simple argument above is that it is clear how to bound the active-set complexity of the method. We formalize this in the following result.
\begin{corollary}\label{cor:complexity}
Consider problem \eqref{eq:problem}, where $f$ is $\mu$-strongly convex with $L$-Lipschitz continuous gradient, and the $g_i$ are proper convex and lower semi-continuous. 
Let Assumption \ref{assump:nondegeneracy} hold for the solution $x^*$. Then the proximal gradient method with a step-size of $1/L$ identifies the active-set  after at most $\kappa \log(2L\norm{x^0 -  x^*}/\delta)$ iterations.
\end{corollary}
\begin{proof}
Using Theorem~\ref{thm:convergence}
and $(1 - 1/\kappa)^k \leq \exp(-k/\kappa)$, we have 
\[
	\norm{x^k - x^*} \leq \exp(-k/\kappa)\norm{x^0 - x^*}.
\]
The proof of Lemma \ref{lem:activeset} shows that the active-set identification occurs whenever the inequality $\norm{x^k - x^*} \leq \delta/2L$ is satisfied. For this to be satisfied, it is sufficient to have
\[
\exp(-k/\kappa)\norm{x^0 - x^*} \le \frac{\delta}{2L}.
\]
Taking the $\log$ of both sides and solving for $k$ gives the result. \qed
\end{proof}

It is interesting to note that this bound only depends logarithmically on $1/\delta$, and that if $\delta$ is quite large we can expect to identify the active-set very quickly. 
This $O(\log(1/\delta))$ dependence is in contrast to the previous result of Liang et al.\ who give a bound of the form $O(1/\sum_{i=1}^n \delta_i^2)$ where $\delta_i$ is the distance of $\nabla_i f$  to the boundary of the subdifferential $\partial g_i$ at $x^*$~\cite[Proposition~3.6]{liang2017activity}. Thus, our bound will typically be tighter as it only depends logarithmically on the single smallest $\delta_i$ (though we make the extra assumption of strong-convexity).
In Section~\ref{sec:motivation}, we considered two specific cases of problem \eqref{eq:problem}, for which we can define $\delta$:
\begin{enumerate}
	\item If the $g_i$ enforce non-negativity constraints, then $\delta = \min_{i \in \mathcal{Z}} \nabla_i f(x^*)$. 
	\item If $g$ is a scaled $\ell_1$-regularizer, then $\delta = \lambda - \max_{i \in \mathcal{Z}}|\nabla_i f(x^*)|$.
\end{enumerate}

In the first case we identify the non-zero variables after $\kappa \log(2L\norm{x^0 - x^*}/\min_{i \in \mathcal{Z}}\nabla_i f(x^*))$ iterations. If the minimum gradient over the active-set at the solution $\delta$ is zero, then we may approach the active-set through the interior of the constraint and the active-set may never be identified (this is the purpose of the nondegeneracy condition). Similarly, for $\ell_1$-regularization this result also gives an upper bound on how long it takes to identify the sparsity pattern.

Above we have bounded the number of iterations before $x_i^k = x_i^*$ for all $i \in \mathcal{Z}$. However, in the non-negative and L1-regularized applications we might also be interested in the number of iterations before we always have $x_i^k \neq 0$ for all $i \not\in\mathcal{Z}$. More generally, the number of iterations before $x_i^k$ for $i\not\in\mathcal{Z}$ are not located at non-smooth or boundary values. It is straightforward to bound this quantity. Let $\Delta = \min_{i \not\in Z}\{|x_i^n - x_i^*|\}$ where $x_i^n$ is the nearest non-smooth or boundary value along dimension $i$. Since~\eqref{eq:linConv} shows that the proximal-gradient method contracts the distance to $x^*$, it cannot set values $x_i^k$ for $i \not\in\mathcal{Z}$ to non-smooth or boundary values once $\norm{x^k - x^*} \leq \Delta$. It follows from~\eqref{eq:linConv} that $\kappa\log(\norm{x^0 - x^*}/\Delta)$ iterations are needed for the values $i \not\in\mathcal{Z}$ to only occur at smooth/non-boundary values.

\section{General step-size}

The previous sections considered a step-size of $1/L$.
In this section we extend our results to handle general constant step-sizes, which leads to a smaller active-set complexity if we use a larger step-size depending on $\mu$.
To do this, we require the following result, which states the generalized convergence rate bound for the proximal gradient method. This result matches the known rate of the gradient method with a constant step-size for solving strictly-convex quadratic problems~\cite[\S 1.3]{bertsekas1999nonlinear}, and the rate of the projected-gradient algorithm with a constant step-size for minimizing strictly-convex quadratic functions over convex sets~\cite[\S 2.3]{bertsekas1999nonlinear}.

\begin{theorem}\label{thm:generalstep}
Consider problem~\eqref{eq:problem}, where $f$ is $\mu$-strongly convex with $L$-Lipschitz continuous gradient, and $g$ is proper convex and lower semi-continuous. 
Then for every iteration $k \ge 1$ of the proximal gradient method with a constant step-size $\alpha > 0$, we have
\begin{equation}
\label{eq:linConv2}
\norm{x^k - x^*} \leq Q(\alpha)^k\norm{x^0 - x^*},
\end{equation}
where $Q(\alpha) := \max \{ |1 - \alpha L|, |1 - \alpha \mu |\}$.
\end{theorem}

We give the proof in the Appendix. Theorem~\ref{thm:convergence} is a special case of Theorem~\ref{thm:generalstep} since $Q(1/ L) = 1 - \mu / L$. Further, Theorem~\ref{thm:generalstep} gives a faster rate if we minimize $Q$ in terms of $\alpha$ to give $\alpha = 2/(L + \mu)$, which yields a faster rate of 
\[
Q\left(\frac{2}{L+\mu}\right)= 1 - \frac{2\mu}{L + \mu} = \frac{L - \mu}{L + \mu}.
\]
 This faster convergence rate for the proximal gradient method may be of independent interest in other settings, and we note that this result does not require $g$ to be separable. We also note that, although the theorem is true for any positive $\alpha$, it is only interesting for $\alpha < 2/L$ since for $\alpha \geq 2/L$ it does not imply convergence.

\begin{lemma}\label{lem:2}
Consider problem \eqref{eq:problem}, where $f$ is $\mu$-strongly convex with $L$-Lipschitz continuous gradient, and the $g_i$ are proper convex and lower semi-continuous.
Let Assumption \ref{assump:nondegeneracy} hold for the solution $x^*$. Then for any proximal gradient method with a constant step-size $0 < \alpha < 2/L$, there exists a $\bar{k}$ such that for all $k > \bar{k}$ we have $x_i^k = x_i^*$ for all $i \in \mathcal{Z}$.
\end{lemma}

We give the proof Lemma~\ref{lem:2} in the Appendix, which shows that we identify the active-set when $\| x^k - x^* \| \le \delta \alpha/3$ is satisfied. Using this result, we prove the following active-set complexity result for proximal gradient methods when using a general fixed step-size (the proof is once again found in the Appendix).

\begin{corollary}\label{cor:2}
Consider problem \eqref{eq:problem}, where $f$ is $\mu$-strongly convex with $L$-Lipschitz continuous gradient (for $\mu < L$), and the $g_i$ are proper convex and lower semi-continuous. 
Let Assumption \ref{assump:nondegeneracy} hold for the solution $x^*$. Then for any proximal gradient method with a constant step-size $\alpha$, such that $0 < \alpha < 2/L$, the active-set will be identified after at most $ \frac{1}{\log (1/Q(\alpha))} \log(3 || x^0 - x^*||/ (\delta \alpha))$ iterations.
\end{corollary}

Finally, we note that as part of a subsequent work we have analyzed the active-set complexity of block coordinate descent methods~\cite{nutini2017}. The argument in that case is similar to the argument presented here. The main modification needed to handle coordinate-wise updates is that we must use a coordinate selection strategy that guarantees that we eventually select all $i \in \mathcal{Z}$ that are not at their optimal values for some finite $k \ge \bar{k}$.


\begin{acknowledgements}
The authors would like to express their thanks to the anonymous referees for their valuable feedback.
\end{acknowledgements}

\section*{Appendix}

{\it Proof of Theorem~\ref{thm:generalstep}}. 
For any $\alpha > 0$, by the non-expansiveness of the proximal operator~\cite[Lem 2.4]{combettes2005} and the fact that $x^*$ is a fixed point of the proximal gradient update for any $\alpha > 0$, we have
\begin{align*}
&\| x^{k+1} - x^*\|^2 \\
&~~= \| \prox{\alpha g}(x^k - \alpha \nabla f(x^k)) - \prox{\alpha g}(x^k - \alpha \nabla f(x^*)) \|^2 \\
&~~\le \| (x^k - \alpha \nabla f(x^k)) - (x^* - \alpha \nabla f(x^*)) \|^2 \\
&~~= \| x^k - x^* - \alpha (\nabla f(x^k) - \nabla f(x^*)) \|^2 \\
&~~= \| x^k - x^* \|^2 - 2 \alpha \langle \nabla f(x^k) - \nabla f(x^*), x^k - x^* \rangle + \alpha^2 \| \nabla f(x^k) - \nabla f(x^*) \|^2.
\end{align*}
By the $L$-Lipschitz continuity of $\nabla f$ and the $\mu$-strong convexity of $f$, we have~\cite[Thm 2.1.12]{Nes04b}
\[
	\langle \nabla f(x^k) - \nabla f(x^*), x^k - x^* \rangle 
	\ge \frac{1}{L+\mu} \| \nabla f(x^k) - \nabla f(x^*) \|^2 + \frac{L \mu}{L + \mu} \| x^k - x^* \|^2,
\]
which yields
\[
\| x^{k+1} - x^*\|^2 
\le \left ( 1 - \frac{2 \alpha L \mu}{L + \mu} \right ) \| x^k - x^* \|^2 + \alpha \left ( \alpha - \frac{2}{L + \mu} \right ) \| \nabla f(x^k) - \nabla f(x^*) \|^2.
\]
Further, by the $\mu$-strong convexity of $f$, we have for any $x, y \in \R^n$~\cite[Thm 2.1.17]{Nes04b},
\[
\langle \nabla f(x) - \nabla f(y)\rangle \ge \mu\|x  - y\|^2,
\]
which by Cauchy-Schwartz gives
\[
	\| \nabla f(x) - \nabla f(y) \| \ge \mu \| x - y \|.
\]
Combining this with the $L$-Lipschitz continuity condition in \eqref{eq:lipschitz}
shows that $\mu \leq L$.  Therefore, for any $\beta \in \R$ (positive or negative) we have
\[
	\beta \| \nabla f(x) - \nabla f(y) \|^2 \le \max \{ \beta L^2, \beta \mu^2 \} \|x - y \|^2.
\]
Thus, for $\beta := \left ( \alpha - \frac{2}{L + \mu}\right )$, we have
\begin{align*}
&\| x^{k+1} - x^*\|^2 \\
&~~\le \left ( 1 - \frac{2 \alpha L \mu}{L + \mu} \right ) \| x^k - x^* \|^2 + \alpha \max \left \{ L^2 \beta, \mu^2 \beta \right \} \| x^k - x^* \|^2 \\
&~~= \max \left \{ \left ( 1 - \frac{2 \alpha L \mu}{L + \mu} \right ) + \alpha L^2 \beta, \left ( 1 - \frac{2 \alpha L \mu}{L + \mu} \right ) + \alpha \mu^2 \beta \right \} \| x^k - x^* \|^2 \\
&~~= \max \left \{ 1 - \frac{2 \alpha L (L + \mu)}{L + \mu} + \alpha^2 L^2, 1 - \frac{2 \alpha \mu (L + \mu)}{L + \mu} + \alpha^2 \mu^2 \right \} \| x^k - x^* \|^2 \\
&~~= \max \left \{ (1 - \alpha L)^2, (1 - \alpha \mu)^2\right \} \| x^k - x^* \|^2 \\
&~~= Q(\alpha)^2 \| x^k - x^* \|^2.
\end{align*}
Taking the square root and applying it repeatedly, we obtain our result. \qed
\bigskip 
\noindent {\it Proof of Lemma~\ref{lem:2}}. 
By the definition of the proximal gradient step and the separability of $g$, for all $i$ we have
\[
	x_i^{k+1} \in \argmin{y}\left\{ \frac{1}{2}\left| y - \left(x^k_i -  \alpha \nabla_i f(x^k)\right) \right|^2 + \alpha g_i(y)\right\}.
\]
This problem is strongly-convex with a unique solution that satisfies
\begin{equation}\label{APPeq:proxSol}
	\frac{1}{\alpha}(x_i^k - y) - \nabla_i f(x^k) \in \partial g_i(y).
\end{equation}
By Theorem~\ref{thm:generalstep} and $\alpha < 2/L$, there exists a minimum finite iterate $\bar{k}$ such that $\norm{x^{\bar{k}}~-~x^*}~\leq~\delta \alpha /3$. Following similar steps as in Lemma~\ref{lem:activeset}, this implies that
\begin{equation}\label{APPeq:itBound}
	-\delta \alpha /3 \leq x_i^k - x_i^* \leq \delta \alpha/3, \quad \text{for all $i$},
\end{equation}
and by the Lipschitz continuity of $\nabla f$, we also have
\begin{equation}\label{APPeq:gBound}
	-\delta \alpha L/3 - \nabla_i f(x^*) \leq - \nabla_i f(x^k) \leq \delta \alpha L/3- \nabla_i f(x^*).
\end{equation}
To complete the proof it is sufficient to show that for any $k \geq \bar{k}$ and $i \in \mathcal{Z}$ that $y = x_i^*$ satisfies \eqref{APPeq:proxSol}. We first show that the left-side is less than the upper limit  $u_i$ of the interval $\partial g_i(x_i^*)$,
\begin{align*}
    \frac{1}{\alpha}(x_i^k - x_i^*) - \nabla_i f(x^k) 
    & \leq \delta/3- \nabla_i f(x^k) & \text{(right-side of~\eqref{APPeq:itBound})}\\
    & \leq \delta(1 + \alpha L)/3 - \nabla_i f(x^*) & \text{(right-side of~\eqref{APPeq:gBound})}\\
    & \leq \delta - \nabla_i f(x^*) & \text{(upper bound on $\alpha$)}\\
    & \leq (u_i + \nabla_i f(x^*)) - \nabla_i f (x^*) & \text{(definition of $\delta$,~\eqref{eq:Delta})}\\
    & \leq u_i.
\end{align*}
Using the left-sides of~\eqref{APPeq:itBound} and~\eqref{APPeq:gBound}, and an analogous sequence of inequalities, we can show that $\frac{1}{\alpha}(x_i^k~-~x_i^*)~-~\nabla_i f(x^k) \geq l_i$, implying that $x_i^*$ solves~\eqref{APPeq:proxSol}. Since the solution to~\eqref{APPeq:proxSol} is unique, this implies the desired result. \qed
\bigskip 
\noindent {\it Proof of Corollary~\ref{cor:2}}. 
By Theorem \ref{thm:generalstep}, we know that the proximal gradient method achieves the following linear convergence rate,
\[
	\norm{x^{k+1} - x^*} \leq Q(\alpha)^k\norm{x^0 - x^*}.
\]
The proof of Lemma \ref{lem:2} shows that the active-set identification occurs whenever the inequality $\norm{x^k - x^*} \leq \delta \alpha/3$ is satisfied. Thus, we want
\[
Q(\alpha)^k\norm{x^0 - x^*} \le \frac{\delta \alpha}{3}.
\]
Taking the $\log$ of both sides, we obtain
\[
	k \log \left (Q(\alpha) \right) + \log \left (\norm{x^0 - x^*} \right ) \le \log \left ( \frac{\delta \alpha}{3} \right ).
\]
Noting that $0 < Q(\alpha) < 1$ so $\log(Q(\alpha)) < 0$, we can rearrange to obtain
\begin{align*}
	k  &\ge \frac{1}{\log \left (Q(\alpha) \right)} \log \left (\frac{\delta \alpha}{3\norm{x^0 - x^*}} \right)
	= \frac{1}{\log(1/Q(\alpha))} \log \left (\frac{3\norm{x^0 - x^*}}{\delta \alpha} \right).
\end{align*}
\qed

\bibliographystyle{spmpsci} 
\bibliography{bib}

\end{document}